\documentclass{amsart}

\usepackage{amssymb}
\usepackage{amsthm}
\usepackage{amsmath}
\usepackage{graphicx}
\usepackage{pifont}
\usepackage{txfonts}
\usepackage[all]{xy}
\usepackage{color}

\theoremstyle{plain}
\newtheorem{theorem}{Theorem}[section]
\newtheorem{proposition}[theorem]{Proposition}
\newtheorem{corollary}[theorem]{Corollary}
\newtheorem{lemma}[theorem]{Lemma}

\newtheorem*{Theorem}{Theorem}
\newtheorem*{Corollary}{Corollary}

\newtheorem*{conjecture*}{Conjecture}
\theoremstyle{definition}
\newtheorem{definition}[theorem]{Definition}
\newtheorem*{definition*}{Definition}
\newtheorem{example}[theorem]{Example}
\newtheorem*{example*}{Example}
\newtheorem*{notation*}{Notation}
\newtheorem*{notation-conv*}{Notation and convention}
\newtheorem*{convention*}{Convention}
\theoremstyle{remark}
\newtheorem{remark}[theorem]{Remark}
\newtheorem*{remark*}{Remark}

\def\co{\colon\thinspace}

\def\la{\langle}
\def\ra{\rangle}

\newcommand{\BT}{\mathbb{T}}

\newcommand{\Z}{\mathbb{Z}}

\newcommand{\C}{\mathbb{C}}

\newcommand{\SL}[1][2]{\mathrm{SL}_{#1}(\C)}
\newcommand{\GL}{\mathrm{GL}_n(\C)}

\newcommand{\I}{\mathbf{1}}
\newcommand{\bm}[1]{\mbox{\boldmath{$#1$}}}

\newcommand{\lk}{\mathop{\mathrm{\ell k}}\nolimits}

\newcommand{\TAP}[2]{\Delta_{#1,\, #2}}

\begin{document}


\title[]{
  Higher dimensional twisted Alexander polynomials for metabelian representations
}

\author{Anh T.~Tran \and Yoshikazu Yamaguchi}

\address{Department of Mathematical Sciences, 
  The University of Texas at Dallas, 
  Richardson, TX 75080, USA}
\email{att140830@utdallas.edu}

\address{Department of Mathematics,
  Akita University,
  1-1 Tegata-Gakuenmachi, Akita, 010-8502, Japan}
\email{shouji@math.akita-u.ac.jp}


\keywords{asymptotic behavior, metabelian representation, 
Reidemeister torsion, twisted Alexander polynomial}
\subjclass[2010]{57M27, 57M50}

\begin{abstract}
  We study the asymptotic behavior of the twisted Alexander polynomial
  for the sequence of $\SL[n]$-representations induced from an irreducible metabelian
  $\SL$-representation of a knot group. 
  We give the limits of the leading coefficients in the asymptotics of
  the twisted Alexander polynomial and related Reidemeister torsion.
  The concrete computations for all genus one two--bridge knots are also presented.
\end{abstract}


\maketitle

\section{Introduction}
The twisted Alexander polynomial was introduced by Lin \cite{Lin} and Wada \cite{Wada94}.
It is a generalization of the classical Alexander polynomial
by a linear representation of a knot group. 
The classical Alexander polynomial contains important topological 
features of knots. 
It has been shown that
the twisted Alexander polynomial gives  refinements of results 
derived from the Alexander polynomial.
On the other hand, the
linear representations of a group are closely related to
the geometric structures of a low--dimensional manifold.
In particular, 
the $\SL$-representations of the fundamental group of a $3$-manifold can be thought of as descriptions of the geometric structures,
especially the hyperbolic structures, of the manifold. We are interested in finding the relationship between 
the twisted Alexander polynomial and the hyperbolic structures of knot exteriors.

It was shown in~\cite{FerrerPorti:HigherDimReidemeister, Muller:AsymptoticsAnalyticTorsion, tranYam:RtosionTroidal}
that the asymptotic behavior of the Reidemeister torsion is related to hyperbolic structures. 
The twisted Alexander polynomial can be regarded  as a kind of Reidemeister torsion by~\cite{KL, Kitano}. 
It is natural to try to relate the asymptotic behavior of the twisted Alexander polynomial
to some features of hyperbolic structures as in~\cite{goda:TAP_hyp_vol}. 

For a metabelian representation of a knot group,
its twisted Alexander polynomial
has been studied in terms of the factorization of this invariant~\cite{BodenFriedl:IV, HirasawaMurasugi:TAP_meta, yamaguchi:TAPmeta}.
In this paper, using such a factorization property,
we wish to investigate the limit of the twisted Alexander polynomial
for a sequence of $\SL[n]$-representations induced by an irreducible metabelian
$\SL$-representation of a knot group and study a relation to degeneration of hyperbolic structures.
Our purpose is to provide the asymptotic behavior 
of {\it the higher dimensional twisted Alexander polynomial}
for irreducible metabelian $\SL$-representations
(see Definition~\ref{def:higher_dim_TAP}).
Then we will study a relation to degeneration of hyperbolic structures
from the viewpoint of the Reidemeister torsion.

Let $E_K$ be the knot exterior $S^3 \setminus N(K)$ of a knot $K$. 
Here $N(K)$ is an open tubular neighbourhood of $K$.
Upon choosing a homomorphism $\rho$ from $\pi_1(E_K)$ into $\SL$,
we have a sequence of homomorphisms $\sigma_n \circ \rho$ from $\pi_1(E_K)$ into $\SL[n]$
by the composition with $\sigma_n \co \SL \to \SL[n]$ 
which is called {\it the $n$-dimensional irreducible representation} of $\SL$.
We can also consider the sequence of twisted Alexander polynomials 
corresponding to the sequence $\rho_n = \sigma_n \circ \rho$.
Our main results are stated as follows:
\begin{Theorem}[Theorem~\ref{thm:main_I}]
  Let $\rho$ be an irreducible metabelian
  $\SL$-representation of $\pi_1(E_K)$. 
  Then there exists a sequence of metabelian $\SL$-representations
  $\psi_1, \ldots, \psi_p$ which gives 
  the limit of the leading coefficient of $\,\log \TAP{K}{\rho_{2N}}(t)$ as
  \[
  \lim_{N \to \infty} \frac{\log \TAP{K}{\rho_{2N}}(t)}{2N}
  = \frac{1}{2p} \sum_{j=1}^p \log \TAP{K}{\psi_j}(t).
  \]
\end{Theorem}
\begin{remark}
  The number $p$ is a divisor of $|\Delta_K(-1)|$ where
  $\Delta_K(t)$ is the Alexander polynomial of $K$.
\end{remark}
Under the assumption that
$\TAP{K}{\psi_j}(1) \not = 0$ for all $j=1, \ldots, p$
in the above theorem, 
the Reidemeister torsion $\BT_{K, \, \rho_{2N}}$ is defined for all $N$.
Then we relate the asymptotic behavior of the twisted Alexander polynomial
to that of the Reidemeister torsion.
The limit of the leading coefficient in $\BT_{K, \, \rho_{2N}}$ is expressed as follows.
\begin{Corollary}[Corollary~\ref{cor:main_I}]
  Taking $t=1$, we can express
  the limit of the leading coefficient of $\,\log |\BT_{K, \, \rho_{2N}}|$ as
  \[
  \lim_{N \to \infty} \frac{\log |\BT_{K, \, \rho_{2N}}|}{2N} 
  = \frac{1}{2p} \sum_{j=1}^p \log |\TAP{K}{\psi_j}(1)|.
  \]
  In particular, the growth order of $\, \log |\BT_{K, \, \rho_{2N}}|$ is the same as $2N$.
\end{Corollary}

For a hyperbolic knot $K$ and its holonomy representation
$\rho_{\mathrm{hol}} \co \pi_1(E_K) \to \SL$,
it was shown in~\cite{FerrerPorti:HigherDimReidemeister, goda:TAP_hyp_vol} that 
the growth order of $\log |\BT_{K, \, \sigma_n \circ \rho_{\mathrm{hol}}}|$ or
$\log |\TAP{K}{\sigma_n \circ \rho_{\mathrm{hol}}}(1)|$ is the same as $n^2$ and 
the leading coefficient converges to the hyperbolic volume of $E_K$ divided by $4\pi$.
The set of $\SL$-representations of $\pi_1(E_K)$ can be regarded as a deformation space of
hyperbolic structures of $E_K$.
From the difference in the growth order of the Reidemeister torsion,
we can say that every irreducible metabelian representation corresponds to
a degenerate hyperbolic structure for any hyperbolic knot.
For examples of this phenomenon, we refer to~\cite{tranYam:RtosionTroidal}.

We also give an explicit description
of the asymptotic behavior of the twisted Alexander polynomial
for metabelian representations of genus one two--bridge knot groups. 
Note that all genus one two--bridge knots are given by $J(2m, \pm 2n)$
illustrated as in Fig.~\ref{fig:doubletwistknot}, where $m$ and $n$ are positive integers. 
\begin{Theorem}[Theorem~\ref{thm:TAP_positive}, \ref{thm:TAP_negative} and
    Corollary~\ref{cor:Rtorsion_positive}, \ref{cor:Rtorsion_negative}]
  Let $\rho$ be an irreducible metabelian $\SL$-representation of
  $\pi_1(E_{J(2m, \pm 2n)})$.
  Then there exists a divisor $q$ of $|\Delta_{J(2m, \pm 2n)}(-1)|$ such that
  \[
  \lim_{N \to \infty} \frac{\log \TAP{J(2m, \pm 2n)}{\rho_{2N}}(t)}{2N}
  = \frac{1}{2q} \log \frac{m^2n^2t^4 + (2m^2n^2 \mp 4mn +1)t^2 + m^2n^2}{(t^2+1)^2} + \frac{1}{2} \log (t^2 + 1).
  \]
  Moreover, the leading coefficient of the logarithm of the higher dimensional Reidemeister torsion
  converges as 
  \[
  \lim_{N \to \infty} \frac{\log |\BT_{{J(2m, \pm 2n)}, \, \rho_{2N}}|}{2N} 
  = \frac{1}{q} \log \frac{2mn \mp 1}{2} + \frac{1}{2} \log 2.
  \]
\end{Theorem}

This paper is organized as follows.
Section~\ref{sec:preliminaries} provides expositions of the twisted Alexander polynomial for a linear representation
and for a sequence of $\SL[n]$-representations induced from an $\SL$-representation
of a knot group. We will also touch a relation between the twisted Alexander polynomial and
the Reidemeister torsion in Subsection~\ref{subsec:relation_TAP_Rtorsion}.
In Section~\ref{sec:general_formula}, our main results are stated and proved.
Section~\ref{sec:twobridge} is devoted to the study of asymptotic behaviors of the twisted Alexander polynomial and
the Reidemeister torsion for genus one two--bridge knots.

\section{Preliminaries}
\label{sec:preliminaries}
We will review the twisted Alexander polynomial and a Lin presentation of
a knot group. Then we will also see a relation between the twisted Alexander polynomial
and the Reidemeister torsion for a knot exterior.
\subsection{Twisted Alexander polynomial}
Throughout the paper, $K$ denotes a knot in $S^3$ and $E_K$ denotes the knot 
exterior $S^3 \setminus N(K)$ where $N(K)$ is an open tubular neighbourhood of $K$.
We use the symbol $\alpha$ to
denote the abelianization homomorphism from a knot group $\pi_1(E_K)$
onto $\la t \ra \simeq H_1(E_K;\Z)$. Here $t$ is the homology class of a meridian of 
the knot $K$.
We follow the definition of the twisted Alexander polynomial in~\cite{Wada94}.

We need a homomorphism $\rho$ from the knot group $\pi_1(E_K)$ into $\GL$,
which is called {\it a $\GL$-representation},
to define the twisted Alexander polynomial for $K$.
For a $\GL$-representation $\rho$ of $\pi_1(E_K)$, we denote by $\Phi_\rho$
the $\Z$-linear extension of the tensor product $\alpha \otimes \rho$ defined as 
\begin{align*}
  \Phi_\rho \co \Z[\pi_1(E_K)] & \to M_n(\C[t^{\pm 1}]) \\
  \sum_{i} a_i \gamma_i & \mapsto  \sum_i a_i \alpha(\gamma_i) \otimes \rho(\gamma_i).
\end{align*}
Here $M_n(\C[t^{\pm 1}])$ is the set of $n \times n$ matrices
whose entries are Laurent polynomials in $\C[t^{\pm 1}]$ and
we identify $M_n(\C[t^{\pm 1}])$ with the tensor product $\C[t^{\pm 1}] \otimes M_n(\C)$
, where $M_n(\C)$ denotes the set of $n \times n$ complex  matrices.

\begin{definition}
  \label{def:TAP_SL2}
  Choose a presentation
  $\pi_1(E_K) = \la g_1, \cdots, g_k \,|\, r_1, \ldots, r_{k-1} \ra$ of deficiency $1$ and
  let $\rho$ be a $\GL$-representation of $\pi_1(E_K)$.
  Suppose that $\Phi_\rho(g_\ell -1) \not = 0$. Then the twisted Alexander polynomial
  is defined as
  \[
  \Delta_{K, \, \rho}(t) =
  \frac{
    \det \left( \Phi_\rho(\frac{\partial r_i}{\partial g_j}) \right)_{
      \substack{ 1 \leq i \leq k-1 \\ 1 \leq j \leq k, j \not =\ell}}
  }{
    \det (\Phi_\rho(g_\ell - 1))
  }
  \]
  where $\partial r_i / \partial g_j$ is the Fox differential of $r_i$ by $g_j$.
\end{definition}

\begin{remark} We mention the well-definedness of the twisted Alexander polynomial without proofs.
  For the details, see~\cite[Theorem~$1$, Corollary~$4$]{Wada94}.
  \begin{itemize}
  \item The twisted Alexander polynomial is independent of the choice of the presentation.
  \item If $\Phi_\rho(g_\ell -1) \not = 0$ and $\Phi_\rho(g_{\ell'} -1) \not = 0$,
    then it holds that
    \[
    \frac{
      \det \left( \Phi_\rho(\frac{\partial r_i}{\partial g_j}) \right)_{
        \substack{ 1 \leq i \leq k-1 \\ 1 \leq j \leq k, j \not =\ell}}
    }{
      \det (\Phi_\rho(g_\ell - 1))
    }
    =
    \pm \frac{
      \det \left( \Phi_\rho(\frac{\partial r_i}{\partial g_j}) \right)_{
        \substack{ 1 \leq i \leq k-1 \\ 1 \leq j \leq k, j \not =\ell'}}
    }{
      \det (\Phi_\rho(g_{\ell'} - 1))
    }.
    \]
  \end{itemize}
\end{remark}

We will mainly consider the twisted Alexander polynomial for 
$\SL[n]$-representations of a knot group. In particular,
we deal with the situation that the
$\SL[n]$-representations $\rho_n$ are induced from
an $\SL$-representation $\rho$ of $\pi_1(E_K)$.

There exists a sequence of homomorphisms $\sigma_{n}$ from $\SL$ into $\SL[n]$,
which is referred as the $n$-dimensional irreducible representations of $\SL$.
This homomorphism $\sigma_n$ is given by the action of $\SL$
on the vector space $V_n$ consisting of homogeneous polynomials $p(x, y)$ with degree $n-1$ as follows:
\[
A \cdot p(x, y) = p(x', y') \quad \hbox{where} \quad
\begin{pmatrix}x' \\ y' \end{pmatrix}
=A^{-1}\begin{pmatrix}x \\ y \end{pmatrix}.
\]
Hence, given an $\SL$-representation $\rho$ of $\pi_1(E_K)$,
there exists the sequence of $\SL[n]$-representations $\sigma_n \circ \rho$.
\begin{remark}
If the eigenvalues of $A \in \SL$ are $\xi^{\pm 1}$,
then $\sigma_n (A)$ has the eigenvalues $\xi^{\mp(2j-1)}$ $(j=1, \ldots, n)$.
One can see this fact from the image by $\sigma_n$ of a diagonal matrix 
with respect to the standard basis
$\{x^{n-1}, x^{n-2}y, \ldots, xy^{n-2}, y^{n-1}\}$
of $V_n$.
\end{remark}

\begin{definition}
  \label{def:higher_dim_TAP}
  Under the assumptions of Definition~\ref{def:TAP_SL2},
  set $\rho_n = \sigma_n \circ \rho$.
  Suppose that $\det (\Phi_{\rho_n}(g_\ell - 1)) \not = 0$.
  Then the twisted Alexander polynomial for $\rho_n$ is defined as
  \[
  \TAP{K}{\rho_n}(t) =
  \frac{
    \det \left( \Phi_{\rho_n}(\frac{\partial r_i}{\partial g_j}) \right)_{
      \substack{1 \leq i \leq k-1 \\ 1 \leq j \leq k, j \not =\ell}}
  }{
    \det (\Phi_{\rho_n}(g_\ell - 1))
  }.
  \]
  We call $\TAP{K}{\rho_n}(t)$ {\it the $n$-dimensional twisted Alexander polynomial for $\rho$}.
\end{definition}
When some $\ell$ satisfies $\det (\Phi_{\rho_n}(g_\ell - 1)) \not = 0$ for all $n$,
we have the sequence of the twisted Alexander polynomial $\TAP{K}{\rho_n}(t)$.

By definition, the twisted Alexander polynomial could be a rational function.
There exists a sufficient condition for it to be a Laurent polynomial.
\begin{proposition}[Proposition~8 in~\cite{Wada94}]
  If a $\GL$-representation $\rho$ of $\pi_1(E_K)$ satisfies that
  there exists an element $\gamma$ in $[\pi_1(E_K), \pi_1(E_K)]$
  such that $\rho(\gamma)$ does not have the eigenvalue $1$, then
  the twisted Alexander polynomial $\Delta_{K, \rho}(t)$ is a Laurent polynomial.
\end{proposition}
Moreover it was shown by~\cite[Theorem~$3$]{KitanoMorifuji:divisibility} that
$\TAP{K}{\rho}(t)$ is a Laurent polynomial for all non--abelian $\SL$-representations $\rho$.
Here the terminology ``non--abelian'' means that
the image $\rho(\pi_1(E_K))$ is not contained in an abelian subgroup in $\SL$.
We will find rational functions in the sequence of $\TAP{K}{\rho_n}(t)$
in Section~\ref{sec:twobridge}.

\subsection{Lin presentation}
In~\cite{Lin}, X.-S. Lin introduced a special presentation of a knot group
by using {\it a free Seifert surface}.
A Seifert surface is called {\it free} if its exterior $S^3 \setminus N(S)$
is a handlebody where $N(S)$ is an open tubular neighbourhood.
This means that $\pi_1(S^3 \setminus N(S))$ is a free group
if a Seifert surface $S$ is free and
the number of generators in $\pi_1(S^3 \setminus N(S))$ coincides with
twice of the genus of $S$.
It is known that every knot $K$ has free Seifert surfaces.
\begin{definition}
  Suppose that $S$ is a free Seifert surface with genus $g$ and 
  $\pi_1(S^3 \setminus N(S))$ is generated by $x_1, \ldots x_{2g}$
  which are the homotopy classes corresponding to the cores of $1$-handles in $S^3 \setminus N(S)$.
  We denote by $\mu$ a meridian of $K$. Then the knot group $\pi_1(E_K)$ is presented as
  \begin{equation}
    \label{eqn:Linpresen}
    \pi_1(E_K) = \la x_1, \ldots, x_{2g}, \mu \,|\, \mu a_i^+ \mu^{-1} = a_i^{-}, i=1, \ldots, 2g \ra
  \end{equation}
  where $a_i^{\pm}$ correspond to loops
  given by pushing up or down the spine $\vee_{i=1}^{2g} a_i$ of $S$
  along the normal direction in $N(S)$.
  We call the presentation~\eqref{eqn:Linpresen} {\it the Lin presentation} of $\pi_1(E_K)$
  associated with $S$.
\end{definition}

\begin{remark} 
  We have the linking number $\lk(x_i, K) = 0$ for every $x_i$ in~\eqref{eqn:Linpresen}.
  This is due to that the representative loop of $x_i$ lies in the outside of $S$.
  This means that each $x_i$ is a commutator in $\pi_1(E_K)$.
\end{remark}

\begin{example} \label{dtk}
  The double twist knot $J(2m ,2n)$ in Hoste-Shanahan's notation~\cite{HosteShanahan}
  is illustrated in Figure~\ref{fig:doubletwistknot}.
  Here both $m$ and $n$ are positive integers.
  The knot $J(2m, 2n)$ corresponds to $K_{-m,\,-n}$ in~\cite{Lin}.
  The right side in Figure~\ref{fig:doubletwistknot} shows a free Seifert surface $S$ of $J(2m, 2n)$.
  We can choose the loops $x_1$ and $x_2$ as generators of $\pi_1(S^3 \setminus N(S))$.
  The spine of $S$ is $a_1 \vee a_2$.
  The loop $a_1^+$ obtained by pushing up $a_1$ is homotopic to $x_1^m$.
  The loops $a_1^-$, $a_2^+$ and $a_2^-$ are homotopic to
  $x_1^mx_2^{-1}$, $x_2^{-n}x_1$ and $x_2^{-n}$ respectively.
  Thus we have the following Lin presentation of $\pi_1(E_{J(2m ,2n)})$:
  \[
  \pi_1(E_{J(2m, 2n)}) =
  \la x_1, x_2, \mu \mid \mu x_1^m \mu^{-1} = x_1^m x_2^{-1}, \, \mu x_2^{-n} x_1 \mu^{-1} =  x_2^{-n} \ra.
  \]
\end{example}

\begin{figure}
  \centering
  \includegraphics[scale=.45]{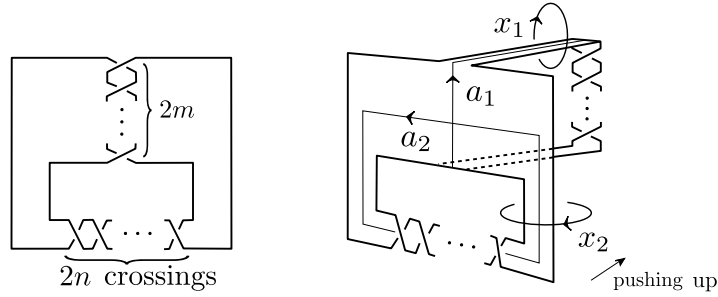}
  \caption{a diagram (left) and a free Seifert surface (right) of $J(2m, 2n)$}
  \label{fig:doubletwistknot}
\end{figure}

\begin{definition}
  Let $\rho$ be an $\SL$-representation of $\pi_1(E_K)$.
  We say that $\rho$ is {\it metabelian}
  if $\rho$ sends the commutator subgroup $[\pi_1(E_K), \pi_1(E_K)]$
  to an abelian subgroup in $\SL$.
  An $\SL$-representation $\rho$ is called {\it irreducible} if
  the action of $\rho(\pi_1(E_K))$ on $\C^2$ does not have any invariant subspaces
  except for $\{\bm{0}\}$ and $\C^2$.
\end{definition}

\begin{proposition}[Proposition~1.1 in~\cite{Nagasato07} and Proposition~2.8 in~\cite{yamaguchi:TAPmeta}]
  \label{prop:metabelianSLII}
  Let $\rho$ be an irreducible metabelian $\SL$-representation of a knot group $\pi_1(E_K)$.
  For any Lin presentation $\pi_1(E_K) = \la x_1, \ldots, x_{2g}, \mu \,|\,
  \mu a_i^{+} \mu^{-} = a_i^-, i=1,\ldots, 2g\ra$,
  there exists a matrix $C \in \SL$ such that 
  \[
  C \rho(x_i) C^{-1} =
  \begin{bmatrix}
    z_i & 0 \\
    0 & z_i^{-1}
  \end{bmatrix}\, (\forall i = 1, \ldots, 2g)
  \quad \hbox{and} \quad
  C\rho(\mu)C^{-1}
  =
  \begin{bmatrix}
    0 & 1 \\
    -1 & 0
  \end{bmatrix}.
  \]
  Here each $z_i$ is a root of unity whose order is a divisor of $\Delta_K(-1)$.
\end{proposition}
\begin{remark}
  It is known that the number of irreducible metabelian $\SL$-representations of $\pi_1(E_K)$
  is equal to $(|\Delta_K(-1)| - 1)/2$.
\end{remark}

\subsection{Relation to the Reidemeister torsion}
\label{subsec:relation_TAP_Rtorsion}
We refer to~\cite{KL, Kitano} for the details on the relation between
the twisted Alexander polynomial and the Reidemeister torsion.
Here we review the Reidemeister torsion from Fox differential
and  restrict our attention to the Reidemeister torsion 
for $E_K$ and
$\SL[n]$-representations $\rho_{n} = \sigma_n \circ \rho$ of $\pi_1(E_K)$.

Choose a presentation $\pi_1(E_K) = \la g_1, \ldots, g_k ,|\, r_1, \ldots, r_{k-1} \ra$ of deficiency $1$.
Let $\widetilde{\rho_n}$ be the linear extension of $\rho_n$ to
the group ring $\Z[\pi_1(E_K)]$.
By a similar argument to~\cite[Proposition~3.1]{Kitano}, one can show the following relation between
the twisted Alexander polynomial and the Reidemeister torsion of a knot exterior.
\begin{proposition}
  \label{prop:Rtorsion_TAP}
  Suppose that $1 \leq \exists \ell \leq k$ such that 
  \[\det(\widetilde{\rho_n}(g_\ell - 1)) \not = 0
  \quad \hbox{and}\quad 
  \det \left(
  \widetilde{\rho_n} \Big(\frac{\partial r_i}{\partial g_j} \Big)
  \right)_{\substack{1 \leq i \leq k-1 \\ 1 \leq j \leq k, j \not =\ell}}
  \not = 0
  \]
  where $\partial r_i / \partial g_j$ is the Fox differential of $r_i$ by $g_j$.
 Then the Reidemeister torsion of $\,\BT_{K, \, \rho}$ is expressed, up to sign, as  
  \begin{equation}
    \label{eqn:def_Rtorsion}
  \BT_{K, \, \rho_n} =
  \frac{
    \det \left(
        \widetilde{\rho_n} \Big(\frac{\partial r_i}{\partial g_j} \Big)
        \right)_{\substack{1 \leq i \leq k-1 \\ 1 \leq j \leq k, j \not =\ell}}
  }{
    \det \widetilde{\rho}(g_\ell - 1)
  }.
  \end{equation}
\end{proposition}

\begin{remark}
  We will use~\eqref{eqn:def_Rtorsion} instead of the definition of $\BT_{K, \, \rho_n}$.
  We touch the well-definedness of $\BT_{K, \, \rho_n}$ without proofs.
  \begin{itemize}
  \item The right hand side of~\eqref{eqn:def_Rtorsion} is equal to
    $\TAP{K}{\rho_n}(1)$.
    Hence it is independent of the choice of $\ell$.
    In the case of $n=2N$,
    the right hand side of~\eqref{eqn:def_Rtorsion} is determined within sign.
  \item
    The assumption in Proposition~\ref{prop:Rtorsion_TAP} means that
    the twisted chain complex used to define the Reidemeister torsion is acyclic.
    For details on the definition of $\BT_{K, \, \rho_n}$, see the review in~\cite{tranYam:RtosionTroidal}.
  \end{itemize}
\end{remark}

\section{Asymptotic behavior for metabelian representations}
\label{sec:general_formula}
Here and subsequently, 
$\rho$ denotes an irreducible metabelian $\SL$-representation of 
a knot group $\pi_1(E_K)$.
Choose a Lin presentation of $\pi_1(E_K)$:
\[
\pi_1(E_K) = 
\la x_1, \dots, x_{2g}, \mu \mid \mu a_i^+ \mu^{-1} = a_i^-, \, i=1, \dots, 2g\ra.
\]
By Proposition~\ref{prop:metabelianSLII}, 
up to conjugation, we may assume that $\rho$ has the following form
\[
\rho(x_i) = 
\begin{bmatrix}
  z_i & 0 \\
  0 & z_i^{-1} \end{bmatrix}, \quad
\rho(\mu) =
\begin{bmatrix}
  0 & 1 \\
  -1 & 0 \end{bmatrix}
\]
where each $z_i$ is a root of unity whose order is a divisor of $\Delta_K(-1)$.
\begin{proposition}
  \label{prop:rep_decompo}
  The induced $\SL[2N]$-representation $\rho_{2N}$ is conjugate to
  the direct sum of metabelian representations as 
  \[\rho_{2N} \underset{\mathrm{conj.}}{\sim} \oplus_{j=1}^N \psi_j\]
  where every $\psi_j$ is a metabelian $\SL$-representation of $\pi_1(E_K)$,
  given by
  \begin{equation}
    \label{eqn:psi}
  \psi_j (x_i) =
  \begin{bmatrix}
    z_i^{1-2j} & 0 \\
    0 & z_i^{2j-1} \end{bmatrix}, \quad
  \psi_j (\mu) = 
  \begin{bmatrix}
    0 & 1 \\
    -1 & 0 \end{bmatrix}.
  \end{equation}
\end{proposition}
\begin{proof}
  Set $W_j = \mathrm{span}_{\C}\la x^{N+j-1}y^{N-j}, x^{N-j}y^{N+j-1}\ra$ in $V_{2N}$
  as the $2$-dimensional subspace
  spanned by $x^{N+j-1}y^{N-j}$ and $x^{N-j}y^{N+j-1}$.
  It follows from a direct computation that
  the restriction of $\rho_{2N}$ to $W_j$ is expressed as
  \[
  \rho_{2N}\big|_{W_j} (x_i) =
  \begin{bmatrix}
    z_i^{1-2j} & 0 \\
    0 & z_i^{2j-1} \end{bmatrix}, \quad
  \rho_{2N}\big|_{W_j} (\mu) = (-1)^{N-j} 
  \begin{bmatrix}
    0 & 1 \\
    -1 & 0 \end{bmatrix}.
  \]
  Taking conjugation by
  $\left(\begin{smallmatrix}
    \sqrt{-1} & 0\\
    0 & -\sqrt{-1}
  \end{smallmatrix}\right)$ if necessary,
  the restriction $\rho_{2N}\big|_{W_j}$ is conjugate to 
  $\psi_j$ given by~\eqref{eqn:psi}.
\end{proof}
By Proposition~\ref{prop:rep_decompo}, one can see that $\rho_{2N}$ has a periodicity
in the direct sum decomposition.
The asymptotic behavior of $\TAP{K}{\rho_{2N}}(t)$ is derived from this periodicity.
\begin{theorem}
  \label{thm:main_I}
  Let $\rho$ be an irreducible metabelian $\SL$-representation of $\pi_1(E_K)$.
  Then there exists metabelian representations $\psi_1, \ldots, \psi_p$ such that
  $\psi_1 = \rho$ and the limit of the leading coefficient of $\,\log \TAP{K}{\rho_{2N}}(t)$ is expressed as
  \begin{equation}
    \label{eqn:main_I}
    \lim_{N \to \infty} \frac{\log \TAP{K}{\rho_{2N}}(t)}{2N}
    = \frac{1}{2p} \sum_{j=1}^{p} \log \TAP{K}{\psi_j}(t).
  \end{equation}
\end{theorem}
\begin{proof}
  The twisted Alexander polynomial has the invariance
  under conjugation of representations.
  We can consider the twisted Alexander polynomials $\TAP{K}{\rho_{2N}}(t)$ for
  $\SL[2N]$-representations $\rho_{2N}$ up to conjugation.
  Proposition~\ref{prop:rep_decompo} shows that
  $\rho_{2N}$ turns into the direct sum of metabelian representations $\psi_j$.
  Under this decomposition of $\rho_{2N}$, it holds that
  \[\TAP{K}{\rho_{2N}}(t) = \prod_{j=1}^N \TAP{K}{\psi_j}(t).\]
  Let $p$ be the l.c.m of orders of the roots of unity $z_1, \ldots, z_{2g}$.
  Note that the l.c.m $p$ is odd
  since the order of $z_i$ is a divisor of the odd integer $\Delta_K(-1)$.
  It is easy seen that $\psi_{j+p} = \psi_j$ for all $j$.
  This implies that $\{\TAP{K}{\psi_j}(t)\}_{j=1, 2, \ldots}$ is
  a periodic sequence of rational functions in $t$
  with period $p$. 
  Hence it follows from~\cite[Lemma~3.11]{Yamaguchi:asymptoticsRtorsion} that
  \[
  \lim_{N \to \infty} \frac{\log \TAP{K}{\rho_{2N}}(t)}{2N}
  = \lim_{N \to \infty} \frac{ \sum_{j=1}^N \log \TAP{K}{\psi_j}(t)}{2N}
  = \frac{1}{2p} \sum_{j=1}^{p} \log \TAP{K}{\psi_j}(t).
  \]
\end{proof}
\begin{remark}
  The period $p$ is the l.c.m. of divisors of $\Delta_K(-1)$. Hence $p$ is also a divisor of $\Delta_K(-1)$.
\end{remark}
Assume that $\TAP{K}{\psi_j}(1) \not = 0$ for all $j=1, \ldots, p$.
Under this assumption, we can also consider the asymptotic behavior of the Reidemeister torsion 
from that of the twisted Alexander polynomial.
\begin{corollary}
  \label{cor:main_I}
  Suppose that $\TAP{K}{\psi_j}(1) \not = 0$ for $1 \leq \forall j \leq p$.
  Then we can define the Reidemeister torsion $\BT_{K, \, \rho_{2N}}$ for all $N$.
  The growth order of $\,\log|\BT_{K, \, \rho_{2N}}|$ is the same as $2N$. 
  The limit of the leading coefficient of $\,\log|\BT_{K, \, \rho_{2N}}|$ is expressed as
  \begin{equation}
    \label{eqn:Rtorsion_general}
    \lim_{N \to \infty} \frac{\log |\BT_{K, \, \rho_{2N}}|}{2N} 
    = \frac{1}{2p} \sum_{j=1}^{p} \log |\TAP{K}{\psi_j}(1)|.
  \end{equation}
\end{corollary}
\begin{proof}
  Proposition~\ref{prop:rep_decompo} shows that 
  $\rho_{2N}$ is conjugate to
  the direct sum $\oplus_{j=1}^N \psi_j$ of $\psi_j$ given by~\eqref{eqn:psi}.
  The twisted Alexander polynomial $\TAP{K}{\rho_{2N}}(t)$ turns into
  the product $\prod_{j=1}^N \TAP{K}{\psi_j}(t)$.
  Since $\psi_{j+p} = \psi_j$ for any $j$, the assumption
  $\TAP{K}{\psi_{j}}(1) \not = 0$
  implies that $\TAP{K}{\rho_{2N}}(1) \not = 0$ for all $N$.
  The eigenvalues of $\rho(\mu)$ are $\pm \sqrt{-1}$.
  One can see that $\det(\rho_{2N}(\mu) - \I) = 2^N$.
  By definition, it holds for all $N$ that  
  \[
  \det \left(
  \widetilde{\rho_{2N}} \Big(\frac{\partial r_k}{\partial x_l} \Big)
  \right)_{\substack{1 \leq k \leq 2g \\ 1 \leq l \leq 2g}}
  \not = 0
  \]
  where $r_k = \mu a_k^+ \mu^{-1} (a_k^-)^{-1}$
  in a Lin presentation.

  It follows from Proposition~\ref{prop:Rtorsion_TAP} that 
  $\BT_{K, \, \rho_{2N}}=\TAP{K}{\rho_{2N}}(1)$
  for any $N$.
  Taking $t=1$, our claim~\eqref{eqn:Rtorsion_general} follows from the equality~\eqref{eqn:main_I}.
\end{proof}

\section{Genus one two--bridge knots}
\label{sec:twobridge}

Genus one two--bridge knots are double twist knots $J(2m, \pm 2n)$ in 
Hoste-Shanahan's notation \cite{HosteShanahan}, 
see also \cite{Lin}. Here $m$ and $n$ are positive integers.
The Alexander polynomial of $K = J(2m, \pm 2n)$ is $$\Delta_K(t) = mn t^2 + (1 \mp 2mn)t +mn.$$

\subsection{The case of \boldmath$K=J(2m,2n)$}
Then $(|\Delta_K(-1)|-1)/2 = 2mn -1$ and 
hence there are $2mn-1$ irreducible metabelian representations $\rho_i \co \pi_1(E_K) \to \SL$,
where $1 \le i \le 2mn-1$.

By Example \ref{dtk} a Lin presentation of $K = J(2m,2n)$ is
\[\pi_1(E_{K})
= \la x_1, x_2, \mu \mid
\mu x_1^m \mu^{-1} = x_1^m x_2^{-1}, \, \mu x_2^{-n} x_1 \mu^{-1} =  x_2^{-n} \ra.\]
Up to conjugation, we may assume that 
\[
\rho_i(x_1) =
\begin{bmatrix}
  \xi^i & 0 \\
  0 & \xi^{-i} 
\end{bmatrix}, \quad
\rho_i(x_2) = 
\begin{bmatrix}
  \xi^{2mi} & 0 \\
  0 & \xi^{-2mi}
\end{bmatrix}, \quad
\rho_i(\mu) = 
\begin{bmatrix}
  0 & 1 \\
  -1 & 0 \end{bmatrix},
\]
where $\xi=e^{2\pi \sqrt{-1}/(4mn-1)}$ and $1 \le i \le 2mn-1$. 
By Proposition~\ref{prop:rep_decompo}, we have 
\[
(\rho_i)_{2N} = \sigma_{2N} \circ \rho_i
\underset{\mathrm{conj.}}{\sim} \bigoplus_{j=1}^{N} \psi_{i,j}
\]
where 
\[
\psi_{i,j}(x_1) = 
\begin{bmatrix}
  \xi^{i(1-2j)} & 0 \\
  0 & \xi^{i(2j-1)} \end{bmatrix}, \quad
\psi_{i,j}(x_2) =
\begin{bmatrix}
  \xi^{2mi(1-2j)} & 0 \\
  0 & \xi^{2mi(2j-1)} \end{bmatrix}, \quad
\psi_{i,j}(\mu) =
\begin{bmatrix}
  0 & 1 \\
  -1 & 0
\end{bmatrix}.
\]

Let $r_1 = \mu x_1^m \mu^{-1} x_2 x_1^{-m}$ and
$r_2 = \mu x_2^{-n} x_1 \mu^{-1} x_2^{n}$. 
We have
\[
\Delta_{{K}, \, \psi_{i,j}}(t)
= \det \begin{bmatrix}
\Phi_{\psi_{i,j}}(\frac{\partial r_1}{\partial x_1}) 
& \Phi_{\psi_{i,j}}(\frac{\partial r_1}{\partial x_2}) \\
\Phi_{\psi_{i,j}}(\frac{\partial r_2}{\partial x_1}) 
& \Phi_{\psi_{i,j}}(\frac{\partial r_2}{\partial x_2})
\end{bmatrix} 
\Big/ \det \Phi_{\psi_{i,j}}(1-\mu),
\]
where 
\begin{align*}
\frac{\partial r_1}{\partial x_1} 
&= \mu(1-x_1^m \mu^{-1}x_2 x_1^{-m}) (1 + x_1 + \cdots + x_1^{m-1}) \\
&= \mu (1 - \mu^{-1})(1 + x_1 + \cdots + x_1^{m-1}),\\
\frac{\partial r_1}{\partial x_2} &= \mu x_1^m \mu^{-1},\\
\frac{\partial r_2}{\partial x_1} &= \mu x_2^{-n},\\
\frac{\partial r_2}{\partial x_2} 
&= \mu (-1 + x_2^{-n} x_1 \mu^{-1} x_2^{n})(1+x_2^{-1}+ \cdots+ x_2^{-(n-1)}) x_2^{-1} \\
&= \mu (-1 + \mu^{-1})(1+x_2^{-1}+ \cdots+ x_2^{-(n-1)})x_2^{-1}.
\end{align*}

For $k \ge 0$ and $b \in \C$ let $\delta_k(b) = 1 + b + \cdots + b^k$. 
Note that $(1+b^{k+1}) \delta_k(b) = \delta_{2k+1}(b)$. 

\begin{proposition}
  \label{prop:main}
  The twisted Alexander polynomial $\TAP{K}{\psi_{i, j}}(t)$ is expressed as follows.
  \begin{itemize}
  \item[(a)] If $\xi^{i(2j-1)} \not= 1$, then 
  \[
    \Delta_{{K}, \, \psi_{i,j}}(t) 
    = ( \xi^{i(2j-1)} )^{1+m-2mn} \left( \delta_{m-1}(\xi^{i(2j-1)}) \delta_{n-1}(\xi^{2m i(2j-1)}) \right)^2 (t^2+1).
  \]
  \item[(b)] If $\xi^{i(2j-1)} = 1$, then 
  \[
  \Delta_{{K},\psi_{i,j}}(t)
  = \frac{m^2n^2t^4 + (2m^2n^2 - 4mn +1)t^2 + m^2n^2}{t^2+1}.
  \]
  \end{itemize}
\end{proposition}

\begin{proof}
  For brevity, we set $a = \xi^{i(1-2j)}$. Then $\psi_{i, j}$ is written as
  \[
  \psi_{i,j}(x_1) =
  \begin{bmatrix}
    a & 0 \\
    0 & a^{-1} \end{bmatrix}, \quad
  \psi_{i,j}(x_2) =
  \begin{bmatrix}
    a^{2m} & 0 \\
    0 & a^{-2m} \end{bmatrix}, \quad 
  \psi_{i,j}(\mu) =
  \begin{bmatrix}
    0 & 1 \\
    -1 & 0
  \end{bmatrix}.
  \]

  It follows that 
  \begin{align*}
    \Phi_{\psi_{i,j}} \left( \frac{\partial r_1}{\partial x_1} \right)
    &=  
    \delta_{m-1}(a)
    \begin{bmatrix}
      -1 & t a^{1-m}  \\
      - t  & - a^{1-m}
    \end{bmatrix},\\
    \Phi_{\psi_{i,j}} \left( \frac{\partial r_1}{\partial x_2} \right)
    &= \begin{bmatrix}
      a^{-m} & 0 \\
      0 & a^m
    \end{bmatrix},\\
    \Phi_{\psi_{i,j}} \left( \frac{\partial r_2}{\partial x_1} \right)
    &= \begin{bmatrix}
      0 & t a^{2mn} \\
      -t a^{-2mn} & 0
    \end{bmatrix},\\
    \Phi_{\psi_{i,j}} \left( \frac{\partial r_2}{\partial x_2} \right)
    &=
    \delta_{n-1}(a^{2m})
    \begin{bmatrix}
      a^{-2mn}  & - t a^{2m}  \\
      t a^{-2mn}  & a^{2m}
    \end{bmatrix}.
  \end{align*}
  Hence, by a direct calculation, we have
  \begin{align*}
    \det \left[
      \Phi_{\psi_{i,j}}\Big(\frac{\partial r_k}{\partial x_l} \Big)
      \right]_{1 \le k, l \le 2}
    &= t^2 - (a^m + a^{1-4mn} ) (a^{2mn} + 1)  \delta_{m-1}(a) \delta_{n-1}(a^{2m})  t^2  \\
    &  + \, a^{1+m-2mn} \left( \delta_{m-1}(a) \delta_{n-1}(a^{2m}) \right)^2 (t^2+1)^2.
  \end{align*}
  Since $\xi^{4mn-1}=1$, we have $a^{4mn-1} = 1$. This implies that 
  \begin{align*} 
    (a^m + a^{1-4mn} ) (a^{2mn} + 1)  \delta_{m-1}(a) \delta_{n-1}(a^{2m}) 
    &= (a^m + 1 ) \delta_{m-1}(a) (a^{2mn} + 1) \delta_{n-1}(a^{2m}) \\
    &= \delta_{2m-1}(a) \delta_{2n-1}(a^{2m}) \\
    &= \delta_{4mn-1}(a) \\
    &=
    \begin{cases}
      \frac{a^{4mn}-1}{a-1} = 1 &\mbox{if } a \not= 1 \\ 
      4mn & \mbox{if } a=1
    \end{cases}.
  \end{align*}
  Hence, we obtain
  \[
  \det \left[
    \Phi_{\psi_{i,j}} \left(\frac{\partial r_k}{\partial x_l} \right)
    \right]_{1 \le k,l \le 2}
  = \begin{cases}
    a^{1+m-2mn} \left( \delta_{m-1}(a) \delta_{n-1}(a^{2m}) \right)^2 (t^2+1)^2 &\mbox{if $a \not= 1$} \\ 
    m^2 n^2 (t^2+1)^2 - (4mn-1)t^2 & \mbox{if $a=1$}.
  \end{cases}
  \]
  The lemma follows, since $\det \Phi_{\psi_{i,j}}(1-\mu) = t^2 +1$.
\end{proof}

\begin{remark}
Note that $a=1$ corresponds to an abelian representation and 
$m^2 n^2 (t^2+1)^2 - (4mn-1)t^2 = \Delta_{J(2m, 2n)}(\sqrt{-1} \, t)\Delta_{J(2m, 2n)}(-\sqrt{-1} \, t)$.
\end{remark}

\begin{lemma}
  \label{lem:main}
  Suppose $\lambda$ is a root of unity whose order is an odd integer $q$. 
  Let $k$ be an positive integer coprime with $q$. Then
  $$\prod_{\substack{1 \le j \le q, \\ j \not= (q+1)/2}} \delta_{k-1}(\lambda^{2j-1}) =1.$$
\end{lemma}

\begin{proof}
  Consider $1 \le j \le q$.
  Note that $\lambda^{2j-1} = 1$ if and only if $j=\frac{q+1}{2}$.
  We have
  \[
  \prod_{\substack{1 \le j \le q, \\ j \not= (q+1)/2}} \delta_{k-1}(\lambda^{2j-1}) 
  = \prod_{\substack{1 \le j \le q, \\ j \not= (q+1)/2}} \frac{\lambda^{k(2j-1)} - 1}{\lambda^{2j-1} - 1}.
  \]
  Since $q = \text{order}(\lambda)$ is an odd integer, for any integer $l$ co-prime with $q$ 
  we have that the map $j \mapsto l(2j-1) \pmod{q}$ gives an bijection from the set
  $\{1 \le j \le q \mid  j \not= (q+1)/2\}$ to $\{1, 2, \cdots, q-1\}$. 
  In particular, we have 
  \[\prod_{\substack{1 \le j \le q, \\ j \not= (q+1)/2}} (\lambda^{l(2j-1)} - 1) 
  = (\lambda -1) (\lambda^2 -1) \cdots (\lambda^{q-1} -1).\]

  Since $k$ is coprime with $q$, 
  the denominator coincides with the numerator in the product of
  $\delta_{k-1}(\lambda^{2j-1}) = 
  \frac{\lambda^{k(2j-1)} - 1}{\lambda^{2j-1} - 1}$. 
  This is our claim.
\end{proof}

Let $q_i = \frac{4mn-1}{\gcd(4mn-1,i)}$. Note that 
$\xi^{i}$ is a root of unity whose order equals to $q_i$. Moreover,  
for $1 \le j \le q_i$ we have $\xi^{i(2j-1)} = 1$ if and only if $j=\frac{q_i+1}{2}$. 

Since the orders of $\xi^i$ and $\xi^{2mi}$ are equal to $q_i$ which is coprime 
with both $m$ and $n$, Lemma~\ref{lem:main} implies that 
\[
\prod_{\substack{1 \le j \le q_i, \\ j \not= (q_i+1)/2}} \delta_{m-1}(\xi^{i(2j-1)}) =1 
\qquad \text{and} \qquad
\prod_{\substack{1 \le j \le q_i, \\ j \not= (q_i+1)/2}} \delta_{n-1}(\xi^{2m i(2j-1)}) =1.
\]

Note that
\[
\prod_{\substack{1 \le j \le q_i, \\ j \not= (q_i+1)/2}} \xi^{i(2j-1)}
= \prod_{1 \le j \le q_i} \xi^{i(2j-1)} = 1.
\]

We obtain the following theorem from the above computations.
\begin{theorem}
  \label{thm:TAP_positive}
  Set $q_i = \frac{4mn-1}{\gcd(4mn-1,i)}$.
  The limit of
  $(\log \TAP{J(2m ,2n)}{(\rho_i)_{2N}}(t) )/(2N)$ is expressed as
  \[
  \lim_{N \to \infty} \frac{\log \TAP{J(2m, 2n)}{(\rho_i)_{2N}}(t)}{2N}
  = \frac{1}{2q_i} \log \frac{m^2n^2t^4 + (2m^2n^2 - 4mn +1)t^2 + m^2n^2}{(t^2+1)^2}
  + \frac{1}{2} \log (t^2 + 1).
  \]
\end{theorem}
\begin{proof}
Since $\Delta_{{K}, \, (\rho_i)_{2N}}(t) = \prod_{j=1}^N \Delta_{{K}, \, \psi_{i,j}}(t)$, 
Proposition \ref{prop:main} implies
\begin{align*}
  &\lim_{N \to \infty} \frac{\log \TAP{J(2m, 2n)}{(\rho_i)_{2N}}(t)}{2N}\\
  &= \frac{1}{2q_i} \sum_{j=1}^{q_i} \log \TAP{J(2m ,2n)}{\psi_{i,j}}(t) \\
  &= \frac{1}{2q_i}
  \left( \log \frac{m^2n^2t^4 + (2m^2n^2 - 4mn +1)t^2 + m^2n^2}{t^2+1} + (q_i - 1) \log (t^2+1) \right)\\
  &= \frac{1}{2q_i} \log \frac{m^2n^2t^4 + (2m^2n^2 - 4mn +1)t^2 + m^2n^2}{(t^2+1)^2} + \frac{1}{2} \log (t^2 + 1).
\end{align*}
\end{proof}

\begin{corollary}
  \label{cor:Rtorsion_positive}
  The leading coefficient of $\,\log |\BT_{J(2m, 2n), \, (\rho_i)_{2N}}|$ converges as follows:
  \begin{align*}
    \lim_{N \to \infty} \frac{\log \BT_{J(2m, 2n), \, (\rho_i)_{2N}}}{2N} 
    &= \frac{1}{q_i} \log \frac{2mn-1}{2} + \frac{1}{2} \log 2 \\
    &= \frac{\gcd(4mn-1,i)}{4mn-1} \log \frac{2mn-1}{2} + \frac{1}{2} \log 2.
  \end{align*}
  In particular,
  the growth order of $\,\log |\BT_{J(2m, 2n), \, (\rho_i)_{2N}}|$ is the same as $2N$.
\end{corollary}
\begin{proof}
  Since $\TAP{J(2m, 2n)}{\psi_{i, j}}(1) \not = 0$ for $1 \leq \forall j \leq q_i$, 
  we can apply Corollary~\ref{cor:main_I} to this situation.
\end{proof}

\subsection{The case of \boldmath$K = J(2m,-2n)$} 
Since $(|\Delta_K(-1)|-1)/2 = 2mn$,
there are $2mn$ irreducible metabelian representations $\rho_i \co \pi_1(E_K) \to \SL$, where $1 \le i \le 2mn$.

A Lin presentation of $K = J(2m,2n)$ is
\[
\pi_1(E_{K}) =
\la x_1, x_2, \mu \mid
\mu x_1^m \mu^{-1} = x_1^m x_2^{-1}, \, \mu x_2^{n} x_1 \mu^{-1} =  x_2^{n} \ra.\]
Up to conjugation, we may assume that 
\[
\rho_i(x_1) = 
\begin{bmatrix}
  \xi^i & 0 \\
  0 & \xi^{-i} 
\end{bmatrix}, \quad
\rho_i(x_2) =
\begin{bmatrix}
  \xi^{2mi} & 0 \\
  0 & \xi^{-2mi}
\end{bmatrix}, \quad
\rho_i(\mu) = 
\begin{bmatrix}
  0 & 1 \\
  -1 & 0 \end{bmatrix},
\]
where $\xi=e^{2\pi \sqrt{-1}/(4mn+1)}$ and $1 \le i \le 2mn$. Then
By Proposition~\ref{prop:rep_decompo} shows that  
\begin{gather*}
  (\rho_i)_{2N} \underset{\mathrm{conj.}}{\sim} \bigoplus_{j=1}^{N} \psi_{i,j}, \\
  \psi_{i,j}(x_1) = 
  \begin{bmatrix}
    \xi^{i(1-2j)} & 0 \\
    0 & \xi^{i(2j-1)} 
  \end{bmatrix}, \quad 
  \psi_{i,j}(x_2) = 
  \begin{bmatrix}
    \xi^{2mi(1-2j)} & 0 \\
    0 & \xi^{2mi(2j-1)} 
  \end{bmatrix}, \quad
  \psi_{i,j}(\mu) = 
  \begin{bmatrix}
    0 & 1 \\
    -1 & 0
  \end{bmatrix}.
\end{gather*}

By a similar argument to the case of $J(2m ,2n)$,
we obtain the asymptotic behavior of the twisted Alexander polynomial 
for $J(2m, -2n)$.
\begin{theorem}
  \label{thm:TAP_negative}
  Let $q_i = \frac{4mn+1}{\gcd(4mn+1,i)}$. 
  For $1 \le i \le 2mn$, it holds that
  \[
  \lim_{N \to \infty} \frac{\log \Delta_{{K}, \, (\rho_i)_{2N}}(t)}{2N} 
  = \frac{1}{2q_i} \log \frac{m^2n^2t^4 + (2m^2n^2 +4mn +1)t^2 + m^2n^2}{(t^2+1)^2}
  + \frac{1}{2} \log (t^2 + 1).
  \]
\end{theorem}
\begin{proof}
  The sequence of $\Delta_{{K}, \, (\rho_i)_{2N}}(t)$ has the period $q_i$.
  The limit of $\frac{\log \Delta_{{K}, \, (\rho_i)_{2N}}(t)}{2N}$ turns into 
  \begin{align*}
    & \frac{1}{2q_i} \sum_{j=1}^{q_i} \log \Delta_{{K}, \, \psi_{i,j}}(t) \\
    &= \frac{1}{2q_i}
    \left( \log \frac{m^2n^2t^4 + (2m^2n^2 + 4mn +1)t^2 + m^2n^2}{t^2+1}
    + (q_i - 1) \log (t^2+1) \right)\\
    &= \frac{1}{2q_i} \log \frac{m^2n^2t^4 + (2m^2n^2 +4mn +1)t^2 + m^2n^2}{(t^2+1)^2}
    + \frac{1}{2} \log (t^2 + 1).
  \end{align*}
\end{proof}

We also have the similar corollary to the case of $J(2m, 2n)$.
\begin{corollary}
  \label{cor:Rtorsion_negative}
  The leading coefficient of $\,\log |\BT_{{K}, \, (\rho_i)_{2N}}|$ converges as follows:
  \begin{align*}
    \lim_{N \to \infty} \frac{\log \BT_{{K}, \, (\rho_i)_{2N}}}{2N}
    &= \frac{1}{q_i} \log \frac{2mn+1}{2} + \frac{1}{2} \log 2 \\
    &= \frac{\gcd(4mn+1,i)}{4mn+1} \log \frac{2mn+1}{2} + \frac{1}{2} \log 2.
  \end{align*}
  In particular,
  the growth order of $\,\log |\BT_{{K}, \, (\rho_i)_{2N}}|$ is the same as $2N$.
\end{corollary}

\begin{remark}
  Since the twist knot $K_n$ in \cite{tranYam:RtosionTroidal} is the mirror image of $J(2, -2n)$, we have
  \[
  \lim_{N \to \infty} \frac{\log \BT_{{K_n}, \, (\rho_i)_{2N}}}{2N} 
  = \frac{\gcd(4n+1,i)}{4n+1} \log \frac{2n+1}{2} + \frac{1}{2} \log 2.
  \]
  Together with \cite[Theorem~3.7]{Yamaguchi:asymptoticsRtorsion} and $\mathrm{order}(\rho_i(\mu)) = 4$,
  one can also obtain the asymptotic behavior of the higher dimensional Reidemeister torsion 
  for the graph manifold obtained by $4$-surgery along $K_n$ as in \cite{tranYam:RtosionTroidal}.
  Compared with~\cite[Theorem~4.4]{tranYam:RtosionTroidal}, it can be seen that 
the coefficient $\gcd(4n+1,i)/(4n+1)$ is determined by the torus knot exterior in the graph manifold obtained by 
$4$-surgery along $K_n$.
\end{remark}

\section*{Acknowledgments}
The first author was partially supported by a grant from the Simons Foundation (\#354595 to AT).
The second author was supported by JSPS KAKENHI Grant Number $26800030$. 


\begin{thebibliography}{MFP14}

\bibitem[BF14]{BodenFriedl:IV}
  H.~U.~Boden and S.~Friedl, \emph{Metabelian $\mathrm{SL}(n, \mathbb{C})$ representations of knot groups IV: twisted Alexander polynomials},
  Math.~Proc.~Cambridge Philos.~Soc. \textbf{156} (2014), 81--97.

\bibitem[God]{goda:TAP_hyp_vol}
H.~Goda, \emph{{T}wisted {A}lexander invariants and {H}yperbolic volume},
  arXiv:1604.07490.
  
\bibitem[HM]{HirasawaMurasugi:TAP_meta}
M.~Hirasawa and K.~Murasugi, \emph{Twisted Alexander polynomials of 2--bridge knots associated to metabelian representations},
  arXiv:0903.1689.

\bibitem[HS04]{HosteShanahan}
J.~Hoste and P.~D. Shanahan, \emph{A formula for the {A}-polynomial of twist
  knots}, J.~Knot Theory Ramifications \textbf{13} (2004), 193--209.

\bibitem[Kit96]{Kitano}
T.~Kitano, \emph{Twisted {A}lexander polynomial and {R}eidemeister torsion},
  Pacific J.~Math. \textbf{174} (1996), 431--442.

\bibitem[KL99]{KL}
P.~Kirk and C.~Livingston, \emph{Twisted {A}lexander {I}nvariants,
  {R}eidemeister torsion, and {C}asson-{G}ordon invariants}, Topology
  \textbf{38} (1999), 635--661.

\bibitem[KM05]{KitanoMorifuji:divisibility}
T.~Kitano and T.~Morifuji, \emph{{D}ivisibility of twisted {A}lexander
  polynomials and fibered knots}, Ann.~Sc.~Norm.~Super.~Pisa Cl.~Sci.~(5)
  \textbf{4} (2005), 179--186.

\bibitem[Lin01]{Lin}
X.-S. Lin, \emph{Representations of knot groups and twisted {A}lexander
  polynomials}, Acta Math.~Sin.~(Engl.~Ser.) \textbf{17} (2001), 361--380.

\bibitem[MFP14]{FerrerPorti:HigherDimReidemeister}
P.~Menal-Ferrer and J.~Porti, \emph{{H}igher dimensional {R}eidemeister torsion
  invariants for cusped hyperbolic $3$-manifolds}, J.~of Topology \textbf{7}
  (2014), 69--119.

\bibitem[M{\"u}l12]{Muller:AsymptoticsAnalyticTorsion}
W.~M{\"u}ller, \emph{{T}he asymptotics of the {R}ay--{S}inger analytic torsion
  of hyperbolic $3$-manifolds}, Metric and Differential Geometry, The Jeff
  Cheeger Anniversary Volume, Progress in Math., vol. 297, Birkh\"auser, 2012,
  pp.~317--352.

\bibitem[Nag07]{Nagasato07}
F.~Nagasato, \emph{{F}initeness of a section of the {${\rm SL}(2,{\mathbb
  C})$}-character variety of knot groups}, Kobe J.~Math. \textbf{24} (2007),
  125--136.

\bibitem[TY]{tranYam:RtosionTroidal}
A.~T.~Tran and Y.~Yamaguchi, \emph{On the asymptotics of the {R}eidemeister
  torsion for exceptional surgeries along twist knots}, preprint 2016, arXiv:1607.06628.

\bibitem[Wad94]{Wada94}
M.~Wada, \emph{Twisted {A}lexander polynomial for finitely presentable groups},
  Topology \textbf{33} (1994), 241--256.

\bibitem[Yam]{Yamaguchi:asymptoticsRtorsion}
Y.~Yamaguchi, \emph{A surgery formula for the asymptotics of the higher
  dimensional {R}eidemeister torsion and {S}eifert fibered spaces}, to appear
  in Indiana Univ.~Math.~J.

\bibitem[Yam13]{yamaguchi:TAPmeta}
\bysame, \emph{On the twisted {A}lexander polynomial for metabelian
  representations into {${\mathrm {SL}}_2 ({\mathbb C})$}}, Topology and its
  Applications \textbf{160} (2013), 1760--1772.

\end{thebibliography}

\newcommand{\noop}[1]{}
\providecommand{\bysame}{\leavevmode\hbox to3em{\hrulefill}\thinspace}
\providecommand{\MR}{\relax\ifhmode\unskip\space\fi MR }
\providecommand{\MRhref}[2]{%
  \href{http://www.ams.org/mathscinet-getitem?mr=#1}{#2}
}
\providecommand{\href}[2]{#2}

\end{document}